\documentclass[12pt]{article}
\usepackage{enumerate}
\usepackage[hyperindex,breaklinks,colorlinks,citecolor=blue,pagebackref]{hyperref}
\usepackage[latin1]{inputenc}
\usepackage{amsmath, amsthm, amsfonts, amssymb}
\usepackage{mathrsfs}
\usepackage{latexsym}
\usepackage{fullpage}
\usepackage[all]{xy}
\usepackage{color}

\newtheorem{theorem}{Theorem}[section]
\newtheorem{lemma}[theorem]{Lemma}
\newtheorem{proposition}[theorem]{Proposition}
\newtheorem{corollary}[theorem]{Corollary}

\newtheorem{definition}[theorem]{Definition}

\theoremstyle{remark}
\newtheorem{remark}{Remark}[section]
\newtheorem{rmks}{Remarks}[section]

\newtheorem*{Examples}{Examples}

\newcommand{\MLR}{\mathsf{MLR}}
\newcommand{\SR}{\mathsf{SR}}
\newcommand{\CR}{\mathsf{CR}}

\newcommand{\dom}{\text{dom}}

\definecolor{purple}{rgb}{.9,0.2,.9}

\newcommand{\cs}{2^\omega}
\newcommand{\uh}{{\upharpoonright}}
\renewcommand{\phi}{\varphi}
\newcommand{\str}{2^{<\omega}}
\newcommand{\binrat}{\mathbb{Q}_2}
\newcommand{\appmu}{\widehat{\mu}}
\newcommand{\leqx}{\preceq}
\newcommand{\halts}{{\downarrow}}

\newcommand{\pre}{\mathsf{Pre}}

\newcommand{\cMLR}{\mathsf{MLR}_{\mathrm{comp}}}
\newcommand{\cNCR}{\mathsf{NCR}_{\mathrm{comp}}}

\newcommand{\Atom}{\mathsf{Atom}}

\title{Effective randomness, strong reductions and Demuth's theorem}
\author{Laurent Bienvenu and Christopher Porter}
\date{} 

\begin{document}

\maketitle

\begin{abstract}
We study generalizations of Demuth's Theorem, which states that the image of a Martin-L\"of random real under a $tt$-reduction
is either computable or Turing equivalent to a Martin-L\"of random real.  We show that Demuth's Theorem holds for Schnorr randomness and computable randomness (answering a question of Franklin), but that it cannot be strengthened by replacing the Turing equivalence in the statement of the theorem with $wtt$-equivalence. We also provide some additional results about the Turing and $tt$-degrees of reals that are random with respect to some computable measure.  
\end{abstract}

\tableofcontents

\newpage

\section{Introduction}

The main topic of this paper is a theorem of Oswald Demuth's concerning effective randomness.  Demuth's theorem is often stated as follows. Given a Martin-L\"of random real~$x$, every non-computable real~$y$ that $x$ $tt$-computes is in turn \emph{Turing} equivalent to a Martin-L\"of random real~$z$. Demuth's original result (in the paper~\cite{Demuth1988}) is written using a slightly outdated terminology (Martin-L\"of random reals are for example called ``non-approximable"), and has sometimes been mistranslated in modern parlance. For example, a stronger version has repeatedly appeared in circulated drafts, talks, and even in some published papers, which asserts that one can even require that $y$ is \emph{wtt}-equivalent to~$z$. This is for example the version given in~\cite{Franklin2008}, where the author further asks:
\begin{itemize}
\item[(a)] whether $z$ can further be required to be $tt$-equivalent to~$y$; and 
\item[(b)] whether Demuth's theorem also holds for Schnorr randomness.
\end{itemize}

In attempting to answer question (a), we realized that even the ``$wtt$-version" of Demuth's theorem was, in fact, false, as we will show below. We will also answer question~(b) positively. 
 
 The rest of the paper is organized as follows. In the first half of the paper, we study the generalization of Demuth's Theorem to different notions of effective randomness. We first (in Section~\ref{sec:background}) provide the necessary background on computable probability measures and their connections to strong reductions, leading to a proof of Demuth's theorem (Section~\ref{sec:demuth}).  Then, we show (Section~\ref{sec:Demuth-for-others}) that the analogue of Demuth's theorem holds for other notions of randomness, namely computable randomness and Schnorr randomness, even though the proof requires some additional effort. In Section~\ref{sec:demuth-wtt}, we study the ``$wtt$-analogue" of Demuth's theorem and show that it fails for all notions of randomness considered in the paper. The last section is a general study of Turing degrees of reals that are random with respect to some computable probability measure. For completeness, the proofs of previously known results as well as technical lemmas needed in our discussion are given in a separate appendix. 

 We assume that the reader is familiar with the basics of computability theory:  computable functions, partial computable functions, computably enumerable sets, Turing functionals, Turing degrees, the Turing jump, and so on (see, for instance, \cite{Soare1987}), as well as the basics of effective randomness (otherwise, we refer the reader to~\cite{DowneyH2010} or ~\cite{Nies2009}). Before moving on to the next section, let us fix some notation and terminology. We denote by $\cs$ the set of infinite binary sequences, also known as Cantor space. We denote the set of finite strings by~$\str$ and the empty string by~$\varnothing$. $\binrat$ is the set of dyadic rationals, i.e., multiples of a negative power of~$2$.  Given $x\in \cs$ and an integer~$n$, $x \uh n$ is the string that consists of the first~$n$ bits of~$x$, and $x(n)$ is the $(n+1)$st of $x$ (so that $x(0)$ is the first bit of $x$).  If $\sigma$ is a string, and $x$ is either a string or an infinite sequence, then $\sigma \leqx x$ means that $\sigma$ is a prefix of~$x$. Given a string~$\sigma$, the \emph{cylinder} $[\sigma]$ is the set of elements of $\cs$ having~$\sigma$ as a prefix. An element~$x \in \cs$ will commonly be identified with the real $\sum_n 2^{-x(n)-1}$,  which belongs to $[0,1]$, and elements of $\cs$ will sometimes be referred to as \emph{reals}.  We can define an ordering $\leq$ on $\cs$ in terms of the standard ordering $\leq$ on [0,1], so that given $x,y\in\cs$, $x\leq y$ if and only if $\sum_n 2^{-x(n)-1}\leq\sum_n 2^{-y(n)-1}$ (caveat: $\leq$ should not be confused with $\preceq$).
 
 
\section{Background}\label{sec:background}

Although Demuth's theorem at first glance does not appear to involve computable probability measures on $\cs$ (and, in fact, Demuth did not explicitly make use of computable probability measures in the original proof of his theorem), the approach we take here will concern the computable measures that are induced by Turing functionals, both total Turing functionals (which are also known as $tt$-functionals) and Turing functionals that are almost total, in a sense we will specify shortly.  To this end, in this section we will review the relevant material on computable measures and Turing functionals, as well as the various notions of effective randomness considered in this paper.

\subsection{On computable measures}


A probability measure on $\cs$ assigns to each Borel set a real in $[0,1]$. It is however sufficient to consider the restriction of probability measures to cylinders. Indeed, Caratheodory's theorem from classical measure theory ensures that a function~$\mu$ defined on cylinders and satisfying for all~$\sigma$ the identity $\mu([\sigma])=\mu([\sigma0])+\mu([\sigma1])$ can be uniquely extended to a probability measure. We can therefore represent measures as functions from strings to reals, where for all~$\sigma\in\str$, $\mu(\sigma)$ is the measure of the cylinder~$[\sigma]$. This concise representation also allows us to talk about \emph{computable} probability measures. 

\begin{definition}
A probability measure $\mu$ on $\cs$ is \emph{computable} if $\sigma \mapsto \mu(\sigma)$ is computable as a real-valued function, i.e. if there is a computable function $\appmu:\str\times\omega\rightarrow\binrat$ such that
\[|\mu(\sigma)-\appmu(\sigma,i)|\leq 2^{-i}\]
for every $\sigma\in\str$ and $i\in\omega$.
 We further say that $\mu$ is \emph{exactly computable} if for all~$\sigma$ $\mu(\sigma) \in \binrat$ and $\sigma \mapsto \mu(\sigma)$ is computable as a function from $\str$ to $\binrat$.
\end{definition}

In what follows~$\lambda$ will refer exclusively to the Lebesgue measure on~$\cs$, where $\lambda(\sigma)=2^{-|\sigma|}$ for each $\sigma\in\str$.  For our purposes, it will be useful to identify several different types of measures.

\begin{definition}
A computable measure~$\mu$ is \emph{positive} if $\mu(\sigma)>0$ for every $\sigma\in\str$. Equivalently, $\mu$ is positive if $\mu(\mathcal{U})>0$ for every non-empty open set~$\mathcal{U}$. 
Moreover,~$\mu$ is \emph{atomic} if there is some real $a\in\cs$ such that $\mu(\{a\})>0$.  In this case, we call $a$ an \emph{atom} of $\mu$ or a $\mu$-atom. In~$\mu$ has no atoms, we call it~\emph{atomless}.  Given an atomic measure $\mu$, the collection of $\mu$-atoms will be denoted $\Atom_\mu$.
\end{definition}

An important result is the following. 

\begin{proposition}[Kautz]
A real $r$ is computable if and only if $r\in\Atom_\mu$ for some computable measure $\mu$.
\end{proposition}

\subsection{Strong reductions and induced measures}

Strong reductions play a central role in the discussion that follows, and more generally in the study of effective randomness.

\begin{definition}
A Turing functional $\Phi: \subseteq \cs\rightarrow\cs$ is
\begin{enumerate}
\item \emph{almost total} if $\lambda(\text{dom}(\Phi))=1$;
\item a \emph{truth-table functional} if $\Phi$ is total;
\item a \emph{weak truth-table functional} if there is some computable function $\phi$ that bounds the use of $\Phi$; and
\item \emph{non-decreasing} if for all $x \leq y$, if both $\Phi(x)$ and $\Phi(y)$ are defined, then $\Phi(x) \leq \Phi(y)$. 
\end{enumerate}
\end{definition}

It is immediate that every truth-table functional is almost total.  Moreover, every truth-table functional is also a weak truth-table functional.
Almost total functionals are important for the study of effective randomness, as we can use them to define computable measures.  

\begin{definition}
Given an almost total functional $\Phi:\cs \rightarrow \cs$, the measure induced by $\Phi$, denoted $\lambda_\Phi$, is defined to be
\[\lambda_\Phi(\mathcal{X})=\lambda(\Phi^{-1}(\mathcal{X}))\]
for every measurable $\mathcal{X}\subseteq\cs$.  
\end{definition}

In general, we can consider the measure induced by a probability measure $\mu$ and a functional $\Phi$, but we have to make the further
restriction that $\Phi$ is almost total \emph{with respect to $\mu$}, i.e., $\mu(\text{dom}(\Phi))=1$.  
Thus we have the following definition.

\begin{definition}
A Turing functional $\Phi: \cs\rightarrow\cs$ is $\mu$-almost total if $\mu(\text{dom}(\Phi))=1$.  Further,
given a measure~$\mu$ on $\cs$ and a $\mu$-almost total functional $\Phi:\cs \rightarrow \cs$, the measure \emph{induced by~$(\mu,\Phi)$}, denoted $\mu_\Phi$, is defined to be
\[\mu_\Phi(\mathcal{X})=\mu(\Phi^{-1}(\mathcal{X}))\]
for every $\mu$-measurable $\mathcal{X}\subseteq\cs$.  
\end{definition}

Intuitively, $\mu_\Phi$ can be computed if $\mu$ and $\Phi$ are given. This is formalized by the following lemma. 

\begin{lemma} \label{lem:comp-induced-measure}
For a given measure~$\mu$ on~$\cs$ and a functional $\Phi: \cs \rightarrow \cs$, the following hold.
\begin{enumerate}
\item If $\mu$ is computable and $\Phi$ is a $\mu$-almost total reduction, then~$\mu_\Phi$ is computable.
\item If $\mu$ is exactly computable and $\Phi$ is a $tt$-reduction, then $\mu_\Phi$ is exactly computable.
\end{enumerate}
\end{lemma}

\begin{proof}
See Appendix \ref{appendix}.
\end{proof}


A useful fact is that the induced measure $\mu_\Phi$ as defined above shares certain features of the original measure $\mu$ as long as the functional $\Phi$ satisfies some additional property:

\begin{lemma} \label{lem:an-induced-measure}
For a given measure~$\mu$ on~$\cs$ and a $\mu$-almost total functional $\Phi: \cs \rightarrow \cs$, the following hold.
\begin{enumerate}
\item If $\mu$ is atomless and $\Phi$ is one-to-one, then~$\mu_\Phi$ is atomless.
\item If $\mu$ is positive and $\Phi$ is onto, then $\mu_\Phi$ is positive.
\end{enumerate}
\end{lemma}

\begin{proof}
Suppose $\mu$ is atomless and $\Phi$ is one-to-one. Then for all~$x$, $\mu_\Phi(\{x\}) = \mu(\Phi^{-1}(\{x\}))$. Since $\Phi$ is one-to-one, $\Phi^{-1}(\{x\})$ is either empty or is a singleton; $\mu$ being atomless, it follows in either case that $\mu(\Phi^{-1}(\{x\}))=0$.

Suppose now that $\mu$ is positive and $\Phi$ is onto. Let $\mathcal{U}$ be a non-empty open set. Since~$\Phi$ is onto (and continuous on its domain), $\Phi^{-1}(\mathcal{U})$ is a non-empty open set (modulo a set of measure~$0$ on which $\Phi$ is not defined), and therefore is has positive $\mu$-measure. 
\end{proof}

\subsection{Notions of randomness}

Although there are many definitions of effective randomness for elements of $\cs$, we restrict our attention to three of the most important definitions:  Martin-L\"of randomness, Schnorr randomness, and computable randomness.  Each of these definitions can be characterized in one of several ways, but we will restrict our attention in this paper to the measure-theoretic formulations of Martin-L\"of randomness and Schnorr randomness and the formulation of computable randomness in terms of certain computable betting strategies known as martingales.  For more details, see \cite{DowneyH2010} or \cite{Nies2009}.

\begin{definition}
Given a computable measure $\mu$, a $\mu$-Martin-L\"of test is a uniformly computable sequence $\{\mathcal{U}_i\}_{i\in\omega}$ of effectively open classes in $\cs$ such that $\mu(\mathcal{U}_i)\leq 2^{-i}$ for every $i\in\omega$.  Further, a real $x$ is $\mu$-Martin-L\"of random if for every $\mu$-Martin-L\"of test $\{\mathcal{U}_i\}_{i\in\omega}$, we have $x\notin\bigcap_{i\in\omega}\mathcal{U}_i$.  The collection of $\mu$-Martin-L\"of random reals will be written as $\MLR_\mu$. 
\end{definition}

As is well-known, for every computable measure $\mu$, there is a universal $\mu$-Martin-L\"of test, which we will write as $\{\widehat{\mathcal{U}}_i\}_{i\in\omega}$.

\begin{proposition}
For every computable measure $\mu$, there is a Martin-L\"of test $\{\widehat{\mathcal{U}}_i\}_{i\in\omega}$ such that $x\in\MLR_\mu$ if and only if $x\notin\bigcap_{i\in\omega}\widehat{\mathcal{U}}_i$.
\end{proposition}

The following result, though rather simple, is very useful, for it allows us to replace a non-positive measure $\mu$ with a positive measure $\nu$ without losing any of the $\mu$-Martin-L\"of random reals.

\begin{lemma}\label{lem:pos-meas}
If $\mu$ is a computable measure, and if $x\in\MLR_\mu$, then $x\in\MLR_{\frac{\mu+\lambda}{2}}$.
\end{lemma}

\begin{proof}
Suppose $x\notin\MLR_{\frac{\mu+\lambda}{2}}$.  Then if $\{\mathcal{U}_i\}_{i\in\omega}$ is a $\frac{\mu+\lambda}{2}$-Martin-L\"of test such that $x\in\bigcap_{i\in\omega}\mathcal{U}_i$, it follows that
\[\frac{\mu(\mathcal{U}_i)+\lambda(\mathcal{U}_i)}{2}\leq 2^{-i}\]
and hence
\[\mu(\mathcal{U}_i)\leq\mu(\mathcal{U}_i)+\lambda(\mathcal{U}_i)\leq 2^{-i+1}.\]
Thus $\{\mathcal{U}_i\}_{i\geq1}$ is a $\mu$-Martin-L\"of test containing $x$, and hence $x\notin\MLR_\mu$.
\end{proof}

We now introduce Schnorr randomness and computable randomness.

\begin{definition}
Given a computable measure $\mu$, a $\mu$-Schnorr test is a uniformly computable sequence $\{\mathcal{U}_i\}_{i\in\omega}$ of effectively open classes in $\cs$ such that $\mu(\mathcal{U}_i)= 2^{-i}$ for every $i\in\omega$.  Further, a real $x$ is $\mu$-Schnorr if for every $\mu$-Schnorr test $\{\mathcal{U}_i\}_{i\in\omega}$, we have $x\notin\bigcap_{i\in\omega}\mathcal{U}_i$.  The collection of $\mu$-Schnorr random reals will be written as $\SR_\mu$.
\end{definition}

\begin{definition}
A computable $\mu$-martingale is a computable function $d:\str\rightarrow\mathbb{R}^{\geq 0}$ such that for every $\sigma\in\str$,
\[\mu(\sigma)d(\sigma)=\mu(\sigma0)d(\sigma0)+\mu(\sigma1)d(\sigma1).\]
Moreover, a computable $\mu$-martingale $d$ succeeds on $x\in\cs$ if 
\[\limsup_{n\rightarrow\infty}d(x\uh n)=\infty.\]
\end{definition}

\begin{definition}
$x\in\cs$ is computably random if there is no computable $\mu$-martingale $d$ that succeeds on $x$.  The collection of $\mu$-computably random reals will be written as $\CR_\mu$.\\
\end{definition}

When we speak about a Martin-L\"of random real (resp.\ computably random, Schnorr random) without specifying the measure, we implicitly mean ``random with respect to the Lebesgue measure". Accordingly, we denote by $\MLR$ (resp. $\CR$, $\SR$) the set of Martin-L\"of random (resp.\ computably random, Schnorr random) reals with respect to Lebesgue measure. \\

It is well known that $\MLR_\mu\subseteq \CR_\mu\subseteq \SR_\mu$ for every computable measure $\mu$ (the inclusions being strict for Lebesgue measure), a result that we will make use of in Section \ref{sec:Demuth-for-others}.  Moreover, we will make use of the following two important results.  The first result shows that every Schnorr random real that is not Martin-L\"of random (and hence every computably random real that is not Martin-L\"of random) must be high, and the second provides a converse to this result:  for every high degree, we can find a computably random real (and hence a Schnorr random real) in that degree.  (Recall that a real $x$ is high if $x'\geq_T\emptyset''$ or equivalently if $x$~computes a function~$g$ which dominates all computable functions).

\begin{theorem}[\cite{NiesST2005}]\label{thm:schnorr-high}
For any computable measure $\mu$, if $x\in\SR_\mu\setminus\MLR_\mu$, then~$x$ must be high.
\end{theorem}

\begin{theorem}[\cite{NiesST2005}]\label{thm:high-cr}
For every high Turing degree $\mathbf{a}$, there is a $x\in\CR$ such that $x\in\mathbf{a}$ (moreover, $x$ can be taken outside $\MLR$).
\end{theorem}

We should note that Theorem \ref{thm:schnorr-high} was proven by Nies, Stephan, and Terwijn only in the case where $\mu$ is Lebesgue measure, but the entire argument goes through when $\mu$ is taken to be an arbitrary computable measure.

\section{On Demuth's theorem}\label{sec:demuth}

Informally, Demuth's Theorem tells us that if we apply an effectively continuous procedure to a Martin-L\"of random real $x$, if the resulting real $y$ has \emph{any} non-trivial computational content whatsoever (i.e.\ if it is not a computable sequence), then from $y$ we can effectively recover a Martin-L\"of random real~$z$.  Formally, the statement is as follows.

\begin{theorem}[Demuth \cite{Demuth1988}]
Let~$x$ be a Martin-L\"of random real. Suppose $x$ $tt$-computes a non-computable real~$y$. Then $y$ is Turing equivalent to some Martin-L\"of random real~$z$. 
\end{theorem}

It is not at all obvious why Demuth's theorem should be true. One informative way to prove it is to break it down into two results, which have been shown independently of Demuth's work. The first result is the well-known ``conservation of randomness" theorem.  Not only do total and almost total functionals induce computable measures (as discussed in the previous section), but according to the conservation of randomness, they also do so in such a way as to map reals random with respect to the original measure to reals random with respect to the induced measure.  
Formally, we have,

\begin{theorem}[Conservation of Martin-L\"of randomness]\label{thm:conservation}
Let $\mu$ be a computable measure and $\Phi$ an almost total functional.  Then $x\in\MLR_\mu$ implies $\Phi(x)\in\MLR_{\mu_\Phi}$.
\end{theorem}

\begin{proof}
See Appendix \ref{appendix}.
\end{proof}

The second result used to derive Demuth's Theorem is sometimes referred to as ``Levin's theorem" or ``the Levin-Kautz theorem" (although Levin proved the theorem with Zvonkin, and independently of Kautz).  

\begin{theorem}[Levin/Zvonkin \cite{LevinZ1970}, Kautz \cite{Kautz1991}] \label{thm:Levin-Kautz}
If $y$ is non-computable and $\mu$-Martin-L\"of random for some computable measure $\mu$, then there is a Martin-L\"of random real $z$ such that $y\equiv_T z$.
\end{theorem}

In the remainder of this section, we will provide a proof of the Levin-Kautz theorem that has not appeared in the effective randomness literature, but which was ``in the air", so to speak.  Before we do so, we should note that the two theorems given above immediately imply Demuth's theorem:  Given a $tt$-functional $\Phi$ and $x\in\MLR$, since $\Phi$ is almost total, by the conservation of randomness, it follows that $\Phi(x)\in\MLR_{\lambda_\Phi}$, and $\lambda_\Phi$ is computable by Lemma~\ref{lem:comp-induced-measure}.  Further, if $\Phi(x)$ is not computable, then by the Levin-Kautz theorem there is some $z\in\MLR$ such that $\Phi(x)\equiv_T z$, thus establishing the result.

Let us now turn to the proof of the Levin-Kautz theorem.  In order to prove the result, we need to prove an auxiliary result that we will refer to as the Kautz conversion procedure.  This result provides a converse to Lemmas~\ref{lem:comp-induced-measure} and \ref{lem:an-induced-measure}, as it shows that any computable measure~$\mu$ can be induced by the Lebesgue measure together with an almost total functional.

\begin{theorem}[The Kautz conversion procedure]\label{thm:Kautz-conversion}
Let~$\mu$ be a computable probability measure. Then there exists a non-decreasing, almost total functional~$\Phi$ such that $\lambda_\Phi=\mu$. Moreover,
\begin{itemize}
\item if $\mu$ is atomless, then $\Phi$ is one-to-one, and
\item if $\mu$ is positive, then $\Phi$ is onto, up to a set of measure zero.
\end{itemize}  
Finally, if $\mu$ is both atomless and positive, then $\Phi$ has an almost total inverse $\Phi^{-1}$ such that $\mu_{\Phi^{-1}}=\lambda$.

\end{theorem}

\begin{proof}
See Appendix \ref{appendix}.
\end{proof}

The next theorem, first proven by Shen, provides a partial converse to the conservation of randomness theorem, stating that every sequence that is random with respect to a computable measure induced by a functional $\Phi$ has a random real in its preimage under $\Phi$.

\begin{theorem}\label{thm:Shen}
Let $\mu$ be a  computable measure, $\Phi$ a $\mu$-almost total functional, and $y\in\cs$.  If $y\in\MLR_{\mu_\Phi}$, then there is some $x\in\MLR_\mu$ such that $\Phi(x)=y$.
\end{theorem}

\begin{proof}
Suppose $y\in\cs$ is such that for all $x\in\cs$ such that $\Phi(x)=y$, $x\notin\MLR_\mu$.  Then if $\{\mathcal{U}_i\}_{i\in\omega}$ is the universal $\mu$-Martin-L\"of test, then consider
\[\mathcal{V}_i=\{z\in\cs: (\forall x)[x\notin\mathcal{U}_i\Rightarrow\Phi(x)\neq z]\}.\]
We claim that $\{\mathcal{V}_i\}_{i\in\omega}$ is a $\mu_\Phi$-Martin-L\"of test.  First, observe that $z\in\mathcal{V}_i$ if and only if $z\notin\Phi(\cs\setminus\mathcal{U}_i)$. Since~$\Phi$ is an almost total Turing functional, the image under~$\Phi$ of a $\Pi^0_1$ class is also a $\Pi^0_1$ class. In particular, $\Phi(\cs\setminus\mathcal{U}_i)$ is a $\Pi^0_1$ class.

To see that $\mu_\Phi(\mathcal{V}_i)\leq 2^{-i}$, since $\cs\setminus\mathcal{U}_i\subseteq\Phi^{-1}(\Phi(\cs\setminus\mathcal{U}_i))$, we have
\[\mu_\Phi(\mathcal{V}_i)=1-\mu_\Phi(\Phi(\cs\setminus\mathcal{U}_i))=1-\mu(\Phi^{-1}(\Phi(\cs\setminus\mathcal{U}_i)))\leq1-\mu(\cs\setminus\mathcal{U}_i)\leq 1-(1-2^{-i})=2^{-i}.\]
Lastly, since for all $x\in\cs$ such that $\Phi(x)=y$, $x\notin\MLR_\mu$, it follows that $x\notin\mathcal{U}_i$ implies $\Phi(x)\neq y$ for every $i$, and so $y\in\mathcal{V}_i$ for every $i$.  Thus $y\notin\MLR_{\mu_\Phi}$.
\end{proof}

We can now prove the Levin-Kautz Theorem.

\begin{proof}[Proof of Theorem \ref{thm:Levin-Kautz}]
Given a non-computable $y$ and a computable measure $\mu$ such that $y\in\MLR_{\mu}$, by Theorem \ref{thm:Kautz-conversion}, there is some non-decreasing almost total functional $\Phi$ such that $\mu=\lambda_\Phi$. 
Since $y\in\MLR_{\lambda_\Phi}$, by Theorem~\ref{thm:Shen}, there is some $z\in\MLR$ such that $\Phi(z)=y$. Moreover, suppose there were another real~$u$ in the preimage of~$y$ under $\Phi$. Since~$\Phi$ is non-decreasing, this would mean that the whole interval $[z,u]$ (or $[u,z]$) is entirely mapped by~$\Phi$ to the singleton~$\{y\}$, and therefore~$y$ would be an atom of $\lambda_\Phi$, hence would be computable, a contradiction.  Therefore, $y$ has a unique preimage~$z$ under~$\Phi$; in other words~$\Phi^{-1}(\{y\})$ is a $\Pi^0_1(y)$-class containing a single element~$z$, hence~$z$ is $y$-computable. We have proven $z \equiv_T y$ and $z\in\MLR$, as wanted.  
\end{proof}



\section{Demuth's theorem for other notions of randomness}\label{sec:Demuth-for-others}

In this section, we consider Demuth's theorem for other notions of effective randomness, namely Schnorr randomness and computable randomness.  
\begin{theorem}\label{thm:schnorr-conserv}
Let $\mu$ be a computable measure and $\Phi$ an almost total functional.  Then $x\in\SR_\mu$ implies $\Phi(x)\in\SR_{\mu_\Phi}$.
\end{theorem}

\begin{proof}
In the proof of the conservation of Martin-L\"of randomness (Theorem \ref{thm:conservation}) in Appendix \ref{appendix}, we show that if $\Phi(x)$ is contained in a $\mu_\Phi$-Martin-L\"of test $\{\mathcal{U}_i\}_{i\in\omega}$, then there is a $\mu$-Martin-L\"of test $\{\mathcal{V}_i\}_{i\in\omega}$ containing $x$.  But in fact, we prove more:  we show that $\mu(\mathcal{V}_i)=\mu_\Phi(\mathcal{U}_i)$.  Thus, if $\{\mathcal{U}_i\}_{i\in\omega}$ is a $\mu_\Phi$-Schnorr test containing $\Phi(x)$, it follows that $\{\mathcal{V}\}_{i\in\omega}$ is a $\mu$-Schnorr test containing~$x$.
\end{proof}

Although we have the proved both the conservation of Martin-L\"of randomness and the conservation of Schnorr randomness, perhaps surprisingly, there is no conservation of computable randomness.  That is, we show that there is a $tt$-functional~$\Phi$ and an $a\in\CR$ such that $\Phi(a)\notin\CR_{\lambda_\Phi}$.

\begin{theorem}
There exists a $tt$-reduction~$\Phi$ that does not preserve computable randomness, that is, for some computably random~$a$, $\Phi(a)$ is not computably random for the measure induced by~$\Phi$. Indeed, one can even construct an example where~$\Phi$ induces the Lebesgue measure. 
\end{theorem}

\begin{proof}
This result follows from the work of Muchnik who proved that Kolmogorov-Loveland randomness is stronger than computable randomness. A Kolmogorov-Loveland random sequence is a sequence that defeats all computable non-monotonic strategies, where a non-monotonic strategy is a betting strategy which at each turn chooses which bit of the sequence (that has not been revealed so far) it will bet on, and then bets on the value of the bit. It was proven in~\cite{MuchnikSU1998} that Kolmogorov-Loveland randomness is strictly stronger than computable randomness, and in~\cite{MerkleMNRS2006} that in the definition Kolmogorov-Loveland randomness, one can assume that only \emph{total} non-monotonic strategies are allowed. Consider therefore a sequence~$a\in \cs$ which is computably random but not Kolmogorov-Loveland random. Let~$S$ be a total non-monotonic strategy that defeats~$a$. Now, define~$\Phi$ to be the $tt$-functional which to a sequence~$x$ associates the sequence~$y$ of the bits of~$x$ seen by~$S$ during the game (in order of appearance). Certainly~$\Phi$ is a $tt$-reduction and induces Lebesgue measure. By definition of Kolmogorov-Loveland randomness, $\Phi(a)$ is not computably random, which proves the result. We can also invoke the stronger result by Kastermans and Lempp~\cite{KastermansL2010} who proved that computable randomness is not closed under computable injective re-ordering of bits (i.e. there exists a computably random real~$a$ and a computable injective function~$f$ such that $a(f(0))a(f(1))\ldots$ is not computably random). 
\end{proof}

The proof of the Levin-Kautz Theorem that we provided in the previous section does not work for Schnorr randomness, since the proof relies upon the existence of a universal Martin-L\"of test, and it is well-known that there is no universal Schnorr test. For computable randomness, the situation is even worse as we have seen that there is no randomness conservation for this notion. As a consequence, we cannot prove Demuth's theorem for these two randomness notions by a direct adaptation of the proof for Martin-L\"of randomness. However, there is an interesting way to overcome the difficulty, which works both for computable randomness and Schnorr randomness. It uses the results of Nies, Stephan, and Terwijn mentioned in Section~\ref{sec:background} (see Theorem~\ref{thm:schnorr-high}): a Schnorr random (resp.\ computably random) real is either Martin-L\"of random, or it is high. Armed with this dichotomy, we get Demuth's theorem for Schnorr randomness and computable randomness almost immediately from Demuth's theorem for Martin-L\"of randomness. In fact, we get a slightly stronger statement that subsumes both, in the sense that it suffices to assume~$x \in \SR$ to get $z \in \CR$ in the conclusion. 

\begin{theorem}[Demuth's theorem for computable and Schnorr randomness]
Let~$x\in\SR$ and let $\Phi$ be a truth-table functional.  If $\Phi(x)=y$ is not computable, then there is some $z\in\CR$ such that $y\equiv_T z$.
\end{theorem}

\begin{proof}
Whether $x \in\CR$ or $x \in \SR$ (the latter being the weaker assumption), the conservation of randomness for Schnorr randomness (Theorem~\ref{thm:schnorr-conserv}) shows that~$y$ is Schnorr random with respect to some computable measure~$\mu$. We now distinguish two cases.

Case 1. If $y$ is not high, then by Theorem~\ref{thm:schnorr-high} it must be $\mu$-Martin-L\"of random. Thus, we can apply the Levin-Kautz theorem (Theorem~\ref{thm:Levin-Kautz}) to get a real~$z \equiv_T y$ that is Martin-L\"of random (hence computably random). 

Case 2. If $y$ is high, then we can directly apply Theorem~\ref{thm:high-cr} to get the existence of some~$z \in \CR$ such that $z \equiv_T y$.

\end{proof}

\section{The failure of Demuth's theorem for $wtt$-reducibility}\label{sec:demuth-wtt}

In the original proof of Demuth's theorem, Demuth shows that in the conclusion of the theorem, one can require $y \equiv_T z$ \emph{and} $y \leq_{tt} z$.

\begin{proposition}
Let~$x$ be a Martin-L\"of random real. Suppose $x$ $tt$-computes a non-computable real~$y$. Then $y$ is Turing equivalent to some Martin-L\"of random real~$z$, and furthermore $y \leq_{tt} z$. 
\end{proposition}

\begin{proof}
The reason is that we can take the $tt$-reduction $\Phi$ from $x$ to $y$ and reorder the truth-tables used by $\Phi$ to define a non-decreasing functional $\hat\Phi$ that induces the same measure as $\Phi$.  More precisely, given a function $h$ that bounds the use of $\Phi$, if we take the table consisting of a column of the $2^{h(n)}$ strings of length $h(n)$ listed in lexicographical order alongside a column of the images of these strings under $\Phi$ (which will be strings of length $n$), we can define a new table by permuting the values in the output column to list them in lexicographical order (leaving the first column fixed), thus yielding a non-decreasing map from strings on length $h(n)$ to strings of length $n$.
\end{proof}

 It is therefore natural to ask whether the reverse reduction, $y \geq_T z$, can also be required to be stronger ($tt$ or $wtt$). We will prove that this is not the case in general, not only for Demuth's original theorem, but also for the versions of Demuth theorem for computable randomness and Schnorr randomness proven in the previous section. 

Although not strictly necessary, our analysis will make use of the so-called complex reals, which were defined by Kjos-Hanssen et al.~\cite{Kjos-HanssenMS2011}. To prove the failure of the $wtt$-version of Demuth's theorem, we will show that (1) if a real~$y$ $wtt$-computes a Martin-L\"of random real~$z$, it must be complex and (2) there is an~$x \in \MLR$ and a real~$y \leq_{tt} x$ such that~$y$ is non-computable but not complex.\\

Complex reals were defined by Kjos-Hanssen et al.\ using plain (or prefix-free) Kolmogorov complexity. We shall define them using another version of Kolmogorov complexity, called \emph{monotone complexity}, which for our purposes is slightly easier to handle. The fact that our definition is equivalent to theirs is proven in Proposition \ref{prop:Km-C-complex} the Appendix.

\begin{definition}\label{def:complex}
A monotone machine is a function $M:\str \rightarrow \str\cup\cs$ such that $M(\sigma_1)\preceq M(\sigma_2)$ for all~$\sigma_1\preceq\sigma_2$ and the set of pairs of strings $(\sigma,\tau)$ with $\tau\preceq M(\sigma)$ is c.e.
Fixing a universal monotone machine $M$, we define the $Km$-complexity of $\tau\in\str$ to be
\[Km(\tau)=\min\{|\sigma|:\tau\preceq M(\sigma)\halts\}.\]
A real~$x$ is said to be complex if there is a computable, non-decreasing, unbounded function $g$ such that $Km(x\uh n)\geq g(n)$ for every $n$.  
\end{definition}
In the sequel we call~\emph{order} a non-decreasing, unbounded function from $\omega$ to $\omega$. And for any order~$g$, $g^{-1}$ is the order defined by
\[
g^{-1}(n) = \min\{k : g(k) \geq n\}
\]

The Levin-Schnorr theorem states that a real~$z$ is Martin-L\"of random (with respect to the Lebesgue measure) if and only if $Km(z \uh n) = n -O(1)$; in particular, Martin-L\"of random reals are complex. Furthermore, any real~$y$ that $wtt$-computes a Martin-L\"of random real is itself complex. This follows from the straightforward fact that complex reals are closed upwards in the $wtt$-degrees.


\begin{lemma}\label{lem:wtt-complexity}
Let $a,b$ be two reals such that $a \geq_{wtt} b$. If~$b$ is complex then so is~$a$. 
\end{lemma}

\begin{proof}
Indeed let~$\phi$ be a computable bound for the use of the~$wtt$-reduction. Suppose~$b$ is complex with order~$g$. Then
\[
g(\phi^{-1}(n)) \leq Km(b \uh \phi^{-1}(n)) \leq Km(a \uh n)
\]
therefore~$a$ is complex via~$g \circ \phi^{-1}$ which is a computable order. 
\end{proof}

The second step of the proof requires more effort. 

\begin{theorem}\label{thm:slowmega1}
There exists $a\in\MLR$ and a non-computable real~$b \leq_{tt} a$ which is not complex. 
\end{theorem}

\begin{proof}
To prove this theorem, we take~$a$ to be Chaitin's $\Omega$ number, where $\Omega:=\sum_{U(\sigma)\halts}2^{-|\sigma|}$, $U$ being a universal prefix-free machine.  It is well-known that $\Omega\in\MLR$ and is left-c.e.\ real, which means that there is a computable sequence of rationals $(\Omega_s)_s$ that converges to $\Omega$ from below. Note that this sequence must converge very slowly, i.e.\ there is no computable function $f$ such that $\Omega\uh n=\Omega_{f(n)}\uh n$ infinitely often, for otherwise we would be able to compress the corresponding initial segments of~$\Omega$. We use the slowness of this approximation to build our sparse real~$b$. We achieve this through the following $tt$-reduction. Let $\Phi: \cs \rightarrow \cs$ be the functional (which we will call the ``slowdown functional") defined for all reals by
\[
\Phi(x)=1^{t_1}01^{t_2}01^{t_3}1\ldots
\]
where the $t_i$ are defined as follows: $t_0=0$ and 
\[
t_i=\min \{s : \Omega_s \geq 0.x \uh i\}
\]
with the convention that if the set on the right-hand side is empty, then $t_i=+\infty$. Thus if some $t_i$ is infinite, then $\Phi(x)=1^{t_1}0\ldots1^{t_k}011111\ldots$ where $t_{k+1}$ is the first $t_i$ to be infinite.

$\Phi$ is clearly a $tt$-reduction.  Moreover, if $a < \Omega$, then there is some $s$ such that $\Omega_s>a\uh i$ for every $i\in\omega$, and hence 
\[\Phi(a)=\sigma(1^k0)^\omega\]
for some $\sigma\in\str$ and $k\in\omega$.  If $a > \Omega$, then there is some $i$ such that $\Omega_s<a\uh i$ for every $s\in\omega$, and hence
\[\Phi(a)=\sigma1^\omega\]
for some $\sigma\in\str$. The interesting case is when $a=\Omega$, for in this case, setting 
\[\Phi(\Omega)=1^{s_1}01^{s_2}01^{s_3}0\ldots,\]
 we know that the function $f$ given by $f(i)=s_i$ grows faster than any computable function, since $\Omega\uh n=\Omega_{f(n)}\uh n$ for every $n\in\omega$. If we set $\Phi(\Omega)=\Omega^*$, then we have
\[Km(\Omega^*\uh f(n))\leq^+ n\]
and hence
\[Km(\Omega^*\uh n)\leq^+ f^{-1}(n),\]
But since $f$ grows faster than any computable order function, then $f^{-1}$ is dominated by all computable order functions.  Thus, there is no computable order function $g$ such that 
\[Km(\Omega^*\uh n)\geq g(n).\]
\end{proof}

We can now prove that the $wtt$-version of Demuth's theorem fails for Martin-L\"of randomness. 

\begin{corollary}\label{cor:ml-demuth-wtt-fails}
There exists $a\in\MLR$ and a non-computable real~$b \leq_{tt} a$ such that there is no $y\in\MLR$ with $y\leq_{wtt}b$.
\end{corollary}

\begin{proof}
By Theorem \ref{thm:slowmega1}, $\Omega^*$ is $tt$-reducible to a Martin-L\"of random real but is not complex.
Hence by Lemma \ref{lem:wtt-complexity}, $\Omega^*$ cannot $wtt$-compute any complex real, and thus $\Omega^*$ cannot $wtt$-compute any $x\in\MLR$.

\end{proof}

There are only countably many reals that are random with respect to the measure induced by the slowdown functional $\Phi$ in the proof of Theorem \ref{thm:slowmega1}  (and all of them are atoms except for $\Omega^*$), but this does not have to be the case, as is shown by the following result.

\begin{proposition}\label{prop:general-slowmega}
There is a computable measure $\mu$ such that $|\MLR_\mu|=2^{\aleph_0}$ and no $x\in\MLR_\mu$ $wtt$-computes any $y\in\MLR$.
\end{proposition}

\begin{proof}
We define a new functional $\Psi$ that on input $a\oplus b$ behaves similarly to the slowdown functional $\Phi$ defined in the proof of Theorem \ref{thm:slowmega1}.
Suppose that $\Phi(a)=1^{t_1}01^{t_2}01^{t_3}0\dotsc1^{t_i}0\dotsc$.  Then we have 
\[\Psi(a\oplus b)=b_0^{t_1}b_1^{t_2}b_2^{t_3}\dotsc b_i^{t_i}\dotsc\]
where $b_i=b(i)$ for every $i$.  Note that $\Psi$ is total, since $\Phi$ is total.  Further, if $B$ is 2-random (i.e. $b\in\MLR^{\emptyset'}$), then $b\in\MLR^\Omega$ and hence $\Omega\oplus b\in\MLR$ by van Lambalgen's Theorem (according to which $A\oplus B\in\MLR\Leftrightarrow A\in\MLR^B$ and $B\in\MLR$ for any $A,B\in\cs$; see \cite{DowneyH2010}, Chapter 6.9).  It follows from the conservation of Martin-L\"of randomness that $\Psi(\Omega\oplus b)$ is random with respect to the induced measure $\lambda_\Psi$.
Moreover, as with $\Omega^*$, $\Psi(\Omega\oplus b)$ is not complex, and thus cannot $wtt$-compute any $y\in\MLR$.
\end{proof}

As we have seen, the $wtt$-generalization of Demuth's Theorem for Martin-L\"of randomness fails quite dramatically. We want to prove that the same is true for computable randomness and Schnorr randomness. It seems that the real $\Omega^*$ constructed in the proof of Theorem~\ref{thm:slowmega1} is so far from complex that it should not even $wtt$-compute a Schnorr random real. Unfortunately, we do not know whether this is the case.
We therefore need to slightly adapt the technique used in the proof of Theorem~\ref{thm:slowmega1}, still keeping the main ideas. 

To prove that the $wtt$-version of Demuth theorem fails for both computable randomness and Schnorr randomness, we will prove the following (stronger) result. 

\begin{theorem}\label{thm:sr-demuth-wtt-fails}
For almost all reals~$a$, there exists a non-computable real~$b \leq_{tt} a$ which does not wtt-compute any Schnorr random real.
\end{theorem} 

This shows in particular that there exists a Martin-L\"of random (hence computably random Schnorr random) real $a$, and a non-computable $b \leq_{tt} a$ which does not $wtt$-compute any Schnorr random real. Therefore the $wtt$-version of Demuth's theorem fails for all three notions of randomness.

To prove Theorem~\ref{thm:sr-demuth-wtt-fails}, we need a few auxiliary facts.  

\begin{lemma}\label{lem:easy-orders}
Let~$f$ be an increasing function that is not dominated by any computable function. Let~$g$ be a computable order function. Then for infinitely many~$n$, 
\[
f(n) < g(f(n+1)).
\] 
\end{lemma}

\begin{proof}
Indeed, if the opposite holds, i.e., $f(n+1) \leq m(f(n))$ for all~$n \geq k$, where~$m$ is the inverse function of~$g$, then it is easy to show by induction that
\[
f(n) \leq m^{(n-k)}(f(k))
\] 
for all~$n \geq k$. The right-hand side of the above expression being a nondecreasing and computable function of $n$, we have a contradiction. 
\end{proof}

\begin{proposition}\label{prop:hyperimm-complexity}
Let~$a$ be a Martin-L\"of random real of hyperimmune degree. Then there is a real $b \leq_{tt} a$ such that~$b$ is not complex. 
\end{proposition}

\begin{proof}
Since~$a$ is of hyperimmune degree, it computes a function~$f$ which is infinitely often above any given computable function~$g$. Let~$\Psi$ be the Turing reduction from~$a$ to~$f$. For all~$n$, define
\[
g(n) = \min\{t: \Psi^a[t](n) \downarrow\}
\]
By the standard conventions on oracle computations, it follows that $g(n) \geq f(n)$ for all~$n$ (as we require that the number of steps for a halting computation always exceeds the output of the computation).  It follows that~$g$ is not dominated by any computable function. Now let~$\Theta$ be the reduction defined by
\[
\Theta(x)=1^{t_0}01^{t_1}01^{t_2}\ldots
\]
with
\[
t_n = \min\{t :\Psi^x[t](n) \downarrow\}
\]
(with the convention that $\Theta(x)=1^{t_0}01^{t_1}01\ldots 1^{t_{i-1}}01111111\ldots$ if $t_i$ is infinite and is the smallest such $t_n$). The definition ensures that~$\Theta$ is total and that
\[
b=\Theta(a)=1^{g(0)}01^{g(1)}01^{g(2)}\ldots
\] 
We need to show that~$b$ is not complex. Let~$h$ be a computable order. Notice that 
\[
\begin{split}
Km(1^{g(0)}01^{g(1)}\ldots 01^{g(n)}01^{g(n+1)})&\leq K(1^{g(0)}01^{g(1)}\ldots 01^{g(n)}0)+O(1) \\
& \leq n \cdot \log g(n)+O(1) \\
& \leq g(n) \log g(n)+O(1),
\end{split}
\]
where the first inequality follows from two facts: (i) $Km(\sigma \tau) \leq K(\sigma)+Km(\tau)+O(1)$ and (ii) $Km(1^k)=O(1)$ for all~$k$. 
By Lemma~\ref{lem:easy-orders} applied to the composition of~$h$ and $(n \mapsto n \log n)^{-1}$, we have for infinitely many~$n$, $g(n) \log g(n)+k < h(g(n+1))$ for any fixed $k\in\omega$ (as $g(n) \log g(n)+k$ is not dominated by any computable function). Thus for infinitely many~$n$,
\[
Km(b \uh g(n+1)) \leq Km(1^{g(0)}01^{g(1)}\ldots 01^{g(n)}01^{g(n+1)}) < h(g(n+1)).
\]
Since this is the case for any order~$h$, it follows that~$b$ is not complex. 
\end{proof}

We are now ready to prove Theorem~\ref{thm:sr-demuth-wtt-fails}. 

\begin{proof}[Proof of Theorem \ref{thm:sr-demuth-wtt-fails}]
Let~$a$ be a random real of hyperimmune but non-high degree. Note that almost all reals have this property. More precisely, any 3-random\footnote{For $n\geq 1$, a real $x$ is $n$-random if and only if  $x\in\MLR^{\emptyset^{(n-1)}}$, that is, $x$ is Martin-L\"of random relative to $\emptyset^{(n-1)}$.} is a real real has this property: any 2-random real has hyperimmune degree, as proven by Kurtz~\cite{Kurtz1981} and no 3-random real is high~\cite[Exercise 8.5.21]{Nies2009}. Then by Proposition~\ref{prop:hyperimm-complexity}, $a$ $tt$-computes a real~$b$ which is not complex. Now suppose $b$ $wtt$-computes a real~$c$. Then since $b$ is not complex, by Lemma~\ref{lem:wtt-complexity}, $c$ is not complex.  In particular, $c$ not Martin-L\"of random (recall that a Martin-L\"of random real~$z$ is s.t. $Km(z \uh n)= n -O(1)$).  Moreover, $c$ is not high, as $a \geq_T b \geq_T c$ and $a$ is not high. Therefore by Theorem \ref{thm:schnorr-high}, if $c$ is not Martin-L\"of random and not high, then $c$ cannot be Schnorr random. 
\end{proof}

\section{On the degrees of random reals}

\subsection{Random Turing degrees}

One consequence of the machinery developed in the previous section is that we can use it to provide an exact characterization of all of the Martin-L\"of random Turing degrees that contain a real that is random with respect to a computable measure but not random with respect to any computable \emph{atomless} measure (recall that a Turing degree is Martin-L\"of random if it contains a Martin-L\"of random real).  Let us establish a few more definitions that will be useful in this section.

\begin{definition}
Let $\cMLR$ be the set of reals $A$ such that $A\in\MLR_\mu$ for some computable measure $\mu$.
\end{definition}

The class $\cMLR$ was, to the best of our knowledge, first considered in ~\cite{FuchsS1977}.  It was later studied in~\cite{MuchnikSU1998}, where elements of $\cMLR$ were referred to as ``natural sequences".

\begin{definition}
Let $\cNCR$ be the set of reals $A$ such that $A\notin\MLR_\mu$ for any computable atomless measure $\mu$.
\end{definition}

The motivation behind the definition of $\cNCR$ comes from the work of Reimann and Slaman (see, for instance, \cite{ReimannS2007} and \cite{ReimannS2008}), who studied the collection of sequences that are not random with respect to \emph{any} atomless measure (computable or otherwise), referring to this class as $\mathsf{NCR}_1$.  Although Reimann and Slaman have established a number of facts about $\mathsf{NCR}_1$, for instance, that it is countable and contains no non-$\Delta^1_1$ reals, a number of questions about the structure of $\mathsf{NCR}_1$ remain open.  $\cNCR$, in contrast, proves to be much easier to characterize.

We will begin by showing that there is at least one $x\in\cMLR\cap\cNCR$.

\begin{proposition}\label{cor:wtt-Demuth}
There is $x\in\cs$ that is random with respect to some computable atomic measure but not random with respect to any computable atomless measure.
\end{proposition}

To prove this proposition, we need one further result.  In Section \ref{sec:background} we saw that if a computable measure $\mu$ is atomless and positive, then if $\Phi$ is an almost total functional such that $\lambda_\Phi=\mu$, then $\Phi^{-1}$ is an almost total functional such that $\mu_{\Phi^{-1}}=\lambda$.  This does not hold in general if $\Phi$ is total, but we can still obtain a measure $\nu$ that is equivalent to $\lambda$, in the sense that $\MLR_\nu=\MLR$.

\begin{proposition}\label{prop:induced-measure-tt}
If $\mu$ is a atomless, computable measure, then there is a non-decreasing $tt$-functional $\Theta$ such that the induced measure $\mu_\Theta$ has the property that
\[
\MLR_{\mu_{\Theta}}=\MLR.
\]
\end{proposition}

\begin{proof}
See Appendix \ref{appendix}.
\end{proof}

\begin{proof}[Proof of Proposition \ref{cor:wtt-Demuth}]
The real constructed in the proof of Theorem~\ref{thm:slowmega1} above, $\Omega^*$, is random with respect to the induced measure $\lambda_\Phi$ (which is clearly atomic), and hence $\Omega^*\in\cMLR$.  Suppose, for sake of contradiction, that $\Omega$ is random with respect to a computable, atomless measure $\mu$.  Then by Proposition \ref{prop:induced-measure-tt}, there is a $tt$-functional $\Theta$ such that $\MLR_{\mu_\Theta}=\MLR$.  Moreover, by the conservation of randomness, it follows that $\Theta(\Omega^*)\in\MLR_{\mu_\Theta}=\MLR$,  but as we proved in Theorem~\ref{thm:slowmega1},  $\Omega^*$ can't even $wtt$-compute any $y\in\MLR$, yielding the desired contradiction.  Thus $\Omega^*\in\cNCR$.
\end{proof}

We can use the idea of this proof to provide a full classification of the Martin-L\"of random Turing degrees containing elements in $\cMLR\cap\cNCR$.  In providing the classification, we will use the following.

\begin{proposition}[\cite{ReimannS2008}, Proposition 5.7]\label{prop:reim-slam}
For $a\in\MLR$ and $b\in\cs$, if $a\equiv_{tt}b$, then $b\notin\cNCR$.
\end{proposition}

\begin{theorem}
Let $\mathbf{a}$ be a Martin-L\"of random Turing degree.  Then 
there is some $a\in\mathbf{a}$ such that $a\in\cMLR\cap\cNCR$ if and only if $\mathbf{a}$ is hyperimmune.
\end{theorem}

\begin{proof}
For the easier direction, suppose $\mathbf{a}$ is hyperimmune-free.  Then given $a\in\mathbf{a}\cap\MLR$, by a well-known result, if $b\equiv_T a$, then $b\equiv_{tt}a$.  Thus for any $b\equiv_T a$, by the conservation of randomness we have $b\in\cMLR$, but by Proposition~\ref{prop:reim-slam}, $b\equiv_{tt}a$ implies that $b\notin\cNCR$, i.e.\ $b$ is random with respect to some atomless measure.  Thus no $b\in\mathbf{a}$ is in $\cMLR\cap\cNCR$.

Now suppose that $\mathbf{a}$ is hyperimmune, and let $a\in\mathbf{a}\cap\MLR$.  We proceed as in the proof of Proposition \ref{prop:hyperimm-complexity}, with a slight modification.  Let $f\in\mathbf{a}$ be a function that is not dominated by any computable function.  Then there is some Turing functional $\Psi$ such that $\Psi^a(n)=f(n)$ for every $n$.  Then, similar to the proof of Proposition~\ref{prop:hyperimm-complexity}, we define a functional $\Gamma$ such that
\[\Gamma(c)=1^{t_0}\;0^{c(0)+1}\;1^{t_1}\;0^{c(1)+1}\;1^{t_2}\;0^{c(2)+1}\dotsc,\]
where $t_i$ is the least $t$ such that $\Psi^c(i)[t]\halts$, unless no such $t$ exists, in which case $t_i=+\infty$.  Note that we code the real $c$ into $\Gamma(c)$ so that if the $(i+1)$st block of 0s in $\Gamma(c)$ has length 1, then $c(i)=0$, and if the $(i+1)$st block of 0s in $\Gamma(c)$ has length 2, then $c(i)=1$.  Thus we have $\Gamma(a)\equiv_Ta$.  Further, by the conservation of randomness, we have $\Gamma(a)\in\cMLR$.  Now let $g:\omega\rightarrow\omega$ be the function such that
\[
\Gamma(a)=1^{g(0)}\;0^{a(0)+1}\;1^{g(1)}\;0^{a(1)+1}\;1^{g(2)}\;0^{a(2)+1}\dotsc
\]
Given the convention that for the least $t$ such that $\Psi^c(i)[t]\halts=k$, we have $k\leq t$, it follows that $f(n)\leq g(n)$, and hence $g(n)$ is not dominated by any computable function.

Now, we verify $\Gamma(a)$ is not complex as before, with the only difference being that we now have to consider the potentially doubled 0s, yielding
\[
Km(1^{g(0)}\; 0^{a(0)+1}\;1^{g(1)}\;0^{a(1)+1}\ldots 1^{g(n)}\;0^{a(n)+1}1^{g(n+1)})\leq2n \cdot \log g(n) \leq g(n) \log g(n).
\]
All the other steps proceed as before, and thus $\Gamma(a)$ is not complex.  Now, assuming that $\Gamma(a)$ is random with respect to some atomless measure, we can argue as in the proof of Proposition \ref{cor:wtt-Demuth} that $\Gamma(a)$ must $tt$-compute a Martin-L\"of random real, contradicting the fact that $\Gamma(a)$ is not complex.  Thus $\Gamma(a)\in\cNCR$.
\end{proof}

Since every hyperimmune degree contains a 1-generic real, and no 1-generic real is Martin-L\"of random with respect to \emph{any} computable measure (as is shown in \cite{MuchnikSU1998}, Theorem 9.10), we have an even stronger dichotomy:  Every hyperimmune-free random degree contains only reals that are random with respect to some computable atomless measure, while every hyperimmune random degree contains reals that are random only with respect to some computable atomic measure as well as reals that aren't random with respect to \emph{any} computable measure.

\subsection{Random computably enumerable sets}

In this last subsection, we will show that the conservation of randomness and related results also have consequences for the study of random computably enumerable sets.  In particular, we show the existence of a computably enumerable set that is random with respect to some computable measure.  This is somewhat surprising, given that computably enumerable sets quite far from Martin-L\"of random.  For instance, every c.e.\ set $x$ has low initial segment complexity:  for every $n$, $K(x\uh n)\leq 2\log(n)+O(1)$.  Despite this behavior, there are c.e.\ sets that are Martin-L\"of random with respect to \emph{some} computable measure, as we now demonstrate (this result was obtained independently by Reimann and Slaman).

\begin{theorem}
There exists a non-computable c.e.\ set $x$ and a computable probability measure~$\mu$ such that~$A$ is random with respect to~$\mu$.  
\end{theorem}

\begin{proof}
Let $(q_n)_{n\in\omega}$ be an effective enumeration of $\binrat$. Let $T: \cs \rightarrow \cs$ be the map defined by
\[
T(x)=\{n \mid q_n < x\}
\]
where we see the input as an infinite binary sequence and the output as a set of integers. Clearly~$T$ is a computable one-to-one map, hence the measure~$\mu$ it induces on~$\cs$ is computable and atomless, and for every random~$x$, $T(x)$ is $\mu$-random. If $x$ is left-c.e.\ by definition of~$T$, $T(x)$ is a c.e.\ set. Therefore, $T(\Omega)$ is both c.e.\ and $\mu$-random. 
\end{proof}

We can also show that there is a non-computable c.e.\ member of $\cMLR\cap\cNCR$.

\begin{theorem}
There is a non-computable c.e.\ set $c$ such that $c\in\MLR_\mu$ for some computable atomic measure $\mu$ but $c\notin\MLR_\nu$ for any computable atomless measure $\nu$.
\end{theorem}

\begin{proof}
To prove this result, we merely need to show that $\Omega^*$, the real constructed in the proof of Theorem \ref{thm:slowmega1}, is left-c.e.\ and then apply the map $T$ defined above to produce a c.e.\ set $C$ that is $tt$-reducible to $\Omega$ (via the composition of $T$ with the slowdown operator $\Phi$ defined in the proof of Theorem \ref{thm:slowmega1}).  It will then follow that $C$ cannot be random with respect to any atomless computable measure, for as we argued in the proof of Proposition~\ref{cor:wtt-Demuth}, this would mean that $C$, and hence $\Omega^*$, can $tt$-compute a 1-random.

To see that $\Omega^*$ is left-c.e., notice that~$\Phi$ is a non-decreasing functional and $\Phi$ is continuous at~$\Omega$. Therefore, for a rational~$q$, we have
\[
q < \Omega^* \Leftrightarrow \exists x \, \left[ x < \Omega \wedge \Phi(x) > q \right]
\]
The right-hand side of the equivalence is a $\Sigma^0_1$ predicate, hence the left cut of $\Omega^*$ is c.e., meaning that $\Omega^*$ is left-c.e.

\end{proof}

Let us make a few remarks. First if a non-computable c.e.\ set is Martin-L\"of random with respect to a computable probability measure, then it must be Turing complete. Indeed, by Demuth's theorem, such a real must be Turing equivalent to a real that is Martin-L\"of random for Lebesgue measure and  Ku$\check{\mathrm{c}}$era~\cite{Kucera1985} proved that a c.e.\ real that can compute a Martin-L\"of random real must necessarily be Turing complete. 

The family of c.e.\ sets that are random for some computable probability measure is therefore not downwards closed in the Turing degrees. However, this family is closed downwards in the\ $tt$-degrees by Demuth's theorem:  given a $tt$-functional $\Phi$ and a c.e.\ set that is random with respect to a computable measure $\mu$, if $\Phi(c)$ is c.e.\, then it is either computable or Turing complete, and in both cases, it will be random with respect to the measure induced by $(\mu,\Phi)$.  It is thus natural to consider whether the family of c.e.\ random sets forms a $tt$-ideal.  As we now show, they do not.

\begin{proposition}
The c.e.\ random sets do not form a~$tt$-ideal. 
\end{proposition}

\begin{proof}
Let $a$ be a left-c.e., Turing incomplete, real and~$x_1$ a left-c.e.\ random real. Set~$x_2=x_1+a$ and notice that~$x_2$ is left-c.e.\ and random as the sum of a random left-c.e.\ real and a left-c.e.\ real~\cite[Chapter 8]{DowneyH2010}. Now convert $x_1$ and $x_2$ into c.e.\ reals via~$T$: $y_1=T(x_1)$ and $y_2=T(x_2)$. Then both $y_1$ and $y_2$ are c.e.\ and random with respect to the measure~$\mu$ induced by~$T$. 

Since~$T$ is a total computable map, its range is a $\Pi^0_1$ class, call it $\mathcal{C}$. Since $T$ is one-to-one, the function $T^{-1}$ is Turing-computable on its domain~$\mathcal{C}$ (indeed for all~$z \in \mathcal{C}$, the set $\{x : T(x)=z\}$ is a $\Pi^0_1(z)$ class containing only one element, hence that element can be computably found when~$z$ is given). It is well-known that a partial functional defined on a $\Pi^0_1$ class can be extended to a total functional. Then let~$S$ be a $tt$-functional which is an extension of $T^{-1}$ to the entire space~$\cs$. 

Now, suppose that the join $y_1\oplus y_2$ is random with respect to some computable measure~$\nu$. Consider the functional $\Psi$ defined by $\Psi(z_1 \oplus z_2)=|S(z_1)-S(z_2)|$. This is a $tt$-functional, and $\Psi(y_1 \oplus y_2)=a$. By the conservation of Martin-L\"of randomness, this means that~$a\in\MLR_{\nu_\Psi}$. This is a contradiction since by the discussion above, an incomplete (left-)c.e.\ real cannot be Martin-L\"of random w.r.t.\ any computable measure.

\end{proof}


\appendix
\section{Appendix}\label{appendix}

\begin{proof}[Proof of Lemma~\ref{lem:comp-induced-measure}]
We proceed inductively as follows:  First, 
\[
\mu_\Phi(\emptyset)=\mu(\Phi^{-1}([\emptyset]))=\mu(\Phi^{-1}(\dom(\Phi)))=1,\]
 since $\Phi$ is almost total.
Now suppose that $\mu_\Phi(\sigma)$ is computable.  Then $\mu_\Phi(\sigma 0)$ and $\mu_\Phi(\sigma 1)$ are both approximable from below, and since $\mu_\Phi(\sigma)=\mu_\Phi(\sigma 0)+\mu_\Phi(\sigma 1)$, it follows that both $\mu_\Phi(\sigma 0)$ and $\mu_\Phi(\sigma 1)$ are approximable from above.  Thus, both are computable.

For the second part, let $\phi$ be a computable function that bounds the use of $\Phi$, i.e.\ if $\Phi(x)=y$, then for every $n\in\omega$, $\Phi^{x\uh\phi(n)}\succeq y\uh n$.
Without loss of generality, we can assume that if $|\sigma|=n$ and $|\tau|<\phi(n)$, then $\Phi^\tau\not\succeq\sigma$.  If we define
\[
\pre_\Phi(\sigma):=\{\tau\in\str:\Phi^{\tau}\succeq\sigma\;\wedge\;(\forall \tau'\preceq \tau)\Phi^{\tau'}\not\succeq\sigma\},
\]
(so that $[\pre_\Phi(\sigma)]=\Phi^{-1}([\sigma])$), it follows that
\[
\pre_\Phi(\sigma)=\{\tau\in 2^{\phi(|\sigma|)}:\Phi^\tau\succeq\sigma\}\]
and thus
\[\mu_\Phi(\sigma)=\mu(\Phi^{-1}([\sigma]))=\mu\bigl(\bigcup_{\tau\in\pre_\Phi(\sigma)}[\tau]\bigr)=\sum_{\tau\in\pre_\Phi(\sigma)}\mu(\tau),\]
which is $\binrat$-valued because $\mu$ is $\binrat$-valued and $\pre_\Phi(\sigma)$ is finite.  Moreover, since we can find, effectively in $\sigma$, the index for $\pre_\Phi(\sigma)$ as a finite set, if follows that $\mu_\Phi$ is a computable function from $\str$ to $\binrat$, and thus is exactly computable.
\end{proof}

\begin{remark}
In the proof of Theorem~\ref{thm:conservation} below, we will have to be careful with the enumeration of our Martin-L\"of tests, and so we will ensure that these tests have nice presentations.  Recall that a set $S\subseteq\str$ is prefix-free if for every $\sigma, \tau\in S$, if $\sigma\preceq\tau$, then $\sigma=\tau$.  Then given a Martin-L\"of test $\{\mathcal{U}_i\}_{i\in\omega}$, we will say that a uniformly computable sequence $\{S_i\}_{i\in\omega}$ of subsets of $\str$ is a prefix-free presentation of $\{\mathcal{U}_i\}_{i\in\omega}$ if we have $\mathcal{U}_i=[S_i]$ for every $i\in\omega$. 
\end{remark}

\begin{proof}[Proof of Theorem~\ref{thm:conservation}]
Suppose that $\Phi(x)\notin\MLR_{\mu_\Phi}$; we will show that $x\notin\MLR_\mu$.  Let $\{\mathcal{U}_i\}_{i\in\omega}$ be a $\mu_\Phi$-Martin-L\"of test such that $\Phi(x)\in\bigcap_{i\in\omega}\mathcal{U}_i$.  We define a $\mu$-Martin-L\"of test $\{\mathcal{V}_i\}_{i\in\omega}$ containing $x$ as follows.  First, let $\{S_i\}_{i\in\omega}$ be a prefix-free presentation of $\{\mathcal{U}_i\}_{i\in\omega}$.  Then we define, for each $i\in\omega$,
\[P_i=\bigcup_{\sigma\in S_i}\pre_\Phi(\sigma).\]
Note that since $S_i$ is prefix-free, for distinct $\sigma_1,\sigma_2\in S_i$, $\pre_\Phi(\sigma_1)\cap\pre_\Phi(\sigma_2)=\emptyset$, and so $\bigcup_{\sigma\in S_i}\pre_\Phi(\sigma)$ is a disjoint union.  Hence
\[
\mu([P_i])=\mu\bigl(\bigcup_{\sigma\in S_i}[\pre_\Phi(\sigma)]\bigr)=\sum_{\sigma\in S_i}\mu([\pre_\Phi(\sigma)])=\sum_{\sigma\in S_i}\mu(\Phi^{-1}([\sigma]))=\mu_\Phi(\mathcal{U}_i).
\]
Now if we set $\mathcal{V}_i:=[P_i]$ for each $i$, we have $\mu(\mathcal{V}_i)=\mu_\Phi(\mathcal{U}_i)$ for each $i$.  In addition, since the collection $\{\mathcal{V}\}_{i\in\omega}$ is definable uniformly from $\{\mathcal{U}\}_{i\in\omega}$, it follows that $\{\mathcal{V}\}_{i\in\omega}$ is a $\mu$-Martin-L\"of test.  Lastly, we must verify that $x\in\bigcap_{i\in\omega}\mathcal{V}_i$.  For each $i$, since $\Phi(x)\in\mathcal{U}_i$, there is some $\sigma\in S_i$ and some least $n\in\omega$ such that $\Phi^{x\uh n}\succeq\sigma$.  Thus $x\uh n\in \text{Pre}_\Phi(\sigma)$, and so it follows that 
$x\uh n\in P_i$ and $x\in\mathcal{V}_i$.

\end{proof}

\begin{proof}[Proof of Theorem~\ref{thm:Kautz-conversion}]
We provide Kautz's proof for completeness.  The key observation in Kautz's proof is that for a given computable measure $\mu$, almost every $x\in[0,1]$ has a binary representation given in terms of $\mu$, which we will refer to as its $\mu$-representation, denoted by Kautz as $\text{seq}_\mu(x)$.  Using this $\mu$-representation, we will define $\Phi$ so that $\Phi(x)=\text{seq}_\mu(x)$.  To compute the $\mu$-representation of $x\in[0,1]$, we make use of what we'll call a $\mu$-partition of [0,1].  A \emph{$\mu$-partition of [0,1] at level $n$} is a collection of $k=2^n$ closed intervals $I_{\sigma_0}, I_{\sigma_1},\dotsc I_{\sigma_{k-1}}$ such that 
\begin{enumerate}
\item $\sigma_0,\sigma_1,\dotsc,\sigma_{k-1}$ is a listing of all strings of length $n$ in lexicographical ordering,
\item $\bigcup_{i=0}^{k-1}I_{\sigma_i}=[0,1]$,
\item $\sup I_{\sigma_i}=\inf I_{\sigma_{i+1}}$ for $0\leq i\leq k-2$, and
\item $\mu(\sigma_i)=\lambda(I_{\sigma_i})$ for $0\leq i\leq k-1$.
\end{enumerate}
We further require that the $\mu$-partition of level $n$ is compatible with the $\mu$-partition of level $n+1$ for every $n$, so that given a string $\sigma$ of length $n$, we have
\[I_\sigma=I_{\sigma0}\cup I_{\sigma1}.\]

Now, given a real $x\in[0,1]$ we can compute its $\mu$-representation $\text{seq}_\mu(x)$ as follows.  To determine the first bit of $\text{seq}_\mu(x)$, we consider the $\mu$-partition of [0,1] at level 1, $I_0\cup I_1.$  Given that $\mu$ is computable but not necessarily exactly computable, we may have to approximate $I_0$ and $I_1$ until we see that $x\in I_0$ or $x\in I_1$, which will occur as long as $x$ is not the right endpoint of $I_0$ (we omit the details).  If $x\in I_0$, the first bit of $\text{seq}_\mu(x)$ is a 0, and if $x\in I_1$, the first bit of $\text{seq}_\mu(x)$ is a 1.  Having determined the first $n$ bits of $\text{seq}_\mu(x)$ by finding $\sigma$ such that $|\sigma|=n$ and $x\in I_\sigma$, we determine whether $x\in I_{\sigma0}$ or $x\in I_{\sigma_1}$ (where $I_{\sigma0}$ and $I_{\sigma1}$ are given by the $\mu$-partition of [0,1] at level $n+1$), and output a 0 or 1 accordingly, as in base case described above.

Thus, if $x$ is not an endpoint of $I_\sigma$ for any $\sigma\in\str$, then $\text{seq}_\mu(x)$ is the unique real $y\in\cs$ such that $x\in I_{y\uh n}$ for every $n$.  Clearly, then $\Phi$ is almost total, and we 
have that 
\[\Phi^{-1}([\sigma])=\{x:\Phi(x)\succ\sigma\}=I_\sigma,\]
so that
\[\lambda_\Phi(\sigma)=\lambda(\Phi^{-1}([\sigma]))=\lambda(I_\sigma)=\mu(\sigma).\]
In the case that $\mu$ is atomless, we have moreover that for every $y\in\cs$, $\lim_{n\rightarrow\infty}\lambda(I_{y\uh n})=0$, which implies that there is a unique $x$ such that $\bigcap_{n\in\omega}I_{y\uh n}=\{x\}$.  Thus, if $\Phi(x_1)=\Phi(x_2)$, we must have $x_1=x_2$.  Next, if $\mu$ is positive, then for every $\sigma\in\str$, $\lambda(I_\sigma)>0$, which means that $\Phi^{-1}(\sigma)$ is non-empty for every $\sigma\in\str$.  Thus, given $y\in\cs$, since 
\[\Phi^{-1}(y\uh n)\supseteq\Phi^{-1}(y\uh(n+1))\]
for every $n$ and each is non-empty, there is some $x$ such that 
\[x\in\bigcap_{n\in\omega}\Phi^{-1}(y\uh n).\]
Lastly, in the case that $\mu$ is both atomless and positive, then 
since $\Phi$ is one-to-one, it has an inverse $\Phi^{-1}$.  Since $\Phi$ is onto up to a set of measure zero, it follows that $\Phi^{-1}$ is almost total.  Given $y\in\cs$ in the range of $\Phi$, i.e.\ $\Phi(x)=y$ for some $x\in\cs$, then $\Phi^{-1}(y)$ can be computed by successively computing $\Phi^{-1}(y\uh n)$ for each $n$ and then intersecting these sets.  More specifically, since
\[
\bigcap_{n\in\omega}\Phi^{-1}(y\uh n)=\{x\},
\]
for each $i$, we will eventually find some $n_i$ such that
\[
z\in\bigcap_{n\leq n_i}\Phi^{-1}(y\uh n)\Rightarrow z\uh i=x\uh i.
\]
Thus we will have $(\Phi^{-1})^{y\uh n_i}\succeq x\uh i$ for every $i$.
\end{proof}

\begin{proposition}\label{prop:Km-C-complex}
The following are equivalent.
\begin{itemize}
\item[(i)] $x$ is complex in the sense of Definition~\ref{def:complex}.
\item[(ii)] There exists a computable order~$h$ such that $C(x \uh n) \geq h(n)$ for all~$n$, $C$ denoting plain Kolmogorov complexity. 
\end{itemize}
\end{proposition}

Item (ii) corresponds to the original definition of complex reals by Kjos-Hanssen et al. 

\begin{proof}
$(i) \rightarrow (ii)$ is trivial as $Km \leq 2C$. For the reverse direction, we use the following result of Kjos-Hanssen et al: a real satisfies (ii) if and only if it $wtt$-computes a sequence of strings $(\sigma_n)$ such that $C(\sigma_n) \geq n$. Now suppose that $x$ is complex, and therefore $wtt$-computes such a sequence $\sigma_n$. Let $\varphi$ be a computable bound on the use of this reduction. We have the following inequalities:
\[
Km(x \uh \phi(n)) \geq^+ C(\sigma_n) - 2\log n \geq n - 2\log n \geq^+ n/2
\]
The second inequality is true by definition of $\sigma_n$. To see that the first one holds, let~$p$ be the shortest $Km$-description of $x \uh \phi(n)$. Using~$p$, one can compute an extension $\tau$ of $x \uh \phi(n)$. Then, specifying~$n$ (for a cost of~$2 \log n+O(1)$ bits), one can retrieve $x \uh \phi(n)$ and therefore compute $\sigma_n$. Since~$Km$ is monotonic, it follows that
\[
Km(x \uh n) \geq^+ \phi^{-1}(n)/2
\]
and the right-hand side is a computable order. 

\end{proof}

\begin{proof}[Proof of Proposition \ref{prop:induced-measure-tt}]
The idea behind the proof is to define a non-decreasing $tt$-functional $\Theta$ such that $\mu_\Theta$ is a generalized Bernoulli measure, i.e. such that for every $n$, there is some $p_n\in[0,1]$ such that 
\[
p_n=\frac{\mu_\Theta(\sigma 0)}{\mu_\Theta(\sigma)}
\]
for every $\sigma\in\str$ of length $n$.  Moreover, we will define $\Theta$ so that 
\[
\biggl|p_n-\frac{1}{2}\biggr|^2\leq 2^{-|\sigma|}
\]
for every $n\in\omega$.  Lastly, we would like to define $\Theta$ in such a way that the resulting values $p_n$ will always be contained in some fixed interval $[\epsilon, 1-\epsilon]$ for $\epsilon\in(0,\frac{1}{2})$; such measures are called \emph{strongly positive}.  Now by the effective version of Kakutani's Theorem (see, for instance, \cite{BienvenuM2009}), given two computable, strongly positive, generalized Bernoulli measures $\mu_1$ (with associated values $p_1,p_2,\dotsc$) and $\mu_2$ (with associated values $q_1,q_2,\dotsc$) such that 
\[
\sum_{i=1}^\infty|p_i-q_i|^2<\infty,
\]
it follows that $\MLR_{\mu_1}=\MLR_{\mu_2}$.  Thus, if we can define such $\Theta$ satisfying the given conditions, then we will have
\[
\sum_{i=1}^\infty\biggl|p_i-\frac{1}{2}\biggr|^2<\infty,
\]
and hence $\MLR_{\mu_\Theta}=\MLR$.  

To define $\Theta$, we sketch the main idea and leave the details to the reader.  To define $p_1$, we look for a finite, prefix-free collection of strings $\{\sigma_1,\dotsc,\sigma_k\}$ such that 
\[
I_{\sigma_1}\cup\dotsc\cup I_{\sigma_k}=[0,\frac{1}{2}-\epsilon_1],
\]
for some $\epsilon_1<\frac{1}{2}$, where $I_\sigma$ is as defined in the proof of Theorem~\ref{thm:Kautz-conversion}  (we can find such a collection effectively because $\mu$ is atomless).  Then we define $\Theta$ so that extensions of each $\sigma_i$ is mapped to extensions of 0 (and reals that extend none of the $\sigma_i$'s are mapped to extensions of 1).  Thus $p_1=\sum_{i\leq k}\mu(\sigma_i)$.

Now we repeat this procedure, partitioning the intervals $[0,\frac{1}{2}-\epsilon_1]$ and $[\frac{1}{2}-\epsilon_1,1]$ each into two intervals, each of which is determined by a finite, prefix-free collection of strings, just as we partitioned the interval [0,1] above, but we must make sure that ratios of the sizes of the components of each partition is the same, i.e.\ the left component of each is $p_2$ times the length of the given interval, where $p_2$ is within $\frac{1}{4}$ of $\frac{1}{2}$.  In so doing, we will get four collections of strings, extensions of which will be mapped to extensions of 00,01,10,11 (depending on which of the four partitions the sequences belong to).  Continuing this procedure, we will eventually define $\Theta$ with the desired properties.
\end{proof}

\bibliographystyle{alpha}
\bibliography{randomness_and_strong_reductions}

\end{document}